\documentclass[preprint,12pt]{elsarticle}

\usepackage[utf8]{inputenc}
\usepackage[english]{babel} 
\usepackage{graphicx}     
\usepackage{hyperref}     
\usepackage{amsthm}
\usepackage{amssymb}
\usepackage{amsmath}
\usepackage{pst-node}
\usepackage{tikz-cd}
\usepackage{mathtools}
\usepackage{MnSymbol}
\usepackage{color,soul}
\usepackage{url}
\usepackage{ragged2e}

\newtheorem{theorem}{Theorem}[section]
\newtheorem{corollary}[theorem]{Corollary}
\newtheorem{lemma}[theorem]{Lemma}
\newtheorem{proposition}[theorem]{Proposition}
\newtheorem{example}[theorem]{Example}

\theoremstyle{definition}
\newtheorem{definition}[theorem]{Definition}
\newtheorem{remark}[theorem]{Remark}

\newcommand{\R}{\mathbb{R}}
\newcommand{\Z}{\mathbb{Z}}

\newcommand{\F}{\mathbb{F}}

\newcommand{\bx}{\mathbf{x}}
\newcommand{\by}{\mathbf{y}}

\newcommand{\vm}{\vec{m}}
\newcommand{\vd}{\vec{d}}

\newcommand{\V}{\mathcal{V}}

\newcommand{\CR}{\mathcal{C}}

\newcommand{\ad}{\mathfrak{a}}

\newcommand{\lcm}{\operatorname{lcm}}
\newcommand{\ord}{\operatorname{ord}}

\newcommand{\coeff}{\operatorname{coeff}}
\newcommand{\Mon}{\operatorname{Mon}}

\newcommand{\Gal}{\operatorname{Gal}}
\newcommand{\Orb}{\operatorname{Orb}}
\newcommand{\Div}{\operatorname{Div}}

\newcommand{\ovl}[1]{\overline{#1}}

\begin{document}

\begin{frontmatter}



\title{Wedderburn decomposition of commutative semisimple group algebras using the Combinatorial Nullstellensatz}

\author{Robert Christian Subroto \\ \texttt{bobby.subroto@ru.nl} \\ \texttt{ORCID: 0000 0001 5534 2655}}
\affiliation{organization={Radboud University},
            addressline={Digital Security},
            city={Nijmegen},
            country={The Netherlands}}

\begin{abstract}
In this paper, we present the simple components of the Wedderburn decomposition of semisimple commutative group algebras over finite abelian groups, which we investigate from a geometric point of view.
We also present the Wedderburn decomposition of semisimple commutative group algebras over finite fields.
\end{abstract}




\end{frontmatter}



\section{Introduction}

For $k$ a field and $G$ an abelian group, the group algebra $k[G]$ is a ring with underlying set consisting of the formal sums $\{ \sum_{g \in G} c_g \cdot g : c_g \in G \}$.
Here for $f := \sum_{g \in G} f_g \cdot g$ and $f':= \sum_{g \in G} f'_g \cdot g'$, addition is defined as 
\begin{align*}
f + f' := \sum_{g \in G} (f_g + f'_g) \cdot g,
\end{align*}
and multiplication as the convolution product:
\begin{align*}
f \cdot f' := \sum_{g \in G} \left( \sum_{u, u' \in G : u \cdot u' = g} f_u f'_{u'}  \right) \cdot g.
\end{align*}

Group algebras have interesting applications in both pure and applied mathematics.
In pure mathematics for example, these group algebras play a very important role in representation theory of finite groups (see \cite{james2001representations}).
The simple components of $k[G]$, assuming that it is semisimple, classify all the irreducible group representations of $G$ over the field $k$ up to isomorhism, effectively providing all the building blocks of all group representations of $G$ over $k$.
In applied mathematics, the algebraic structure of commutative group algebras over finite fields are often used in designing and analyzing cryptographic schemes in both symmetric and asymmetric cryptography.
In symmetric cryptography, these group algebras are used in the designs of for example {\sc AES} \cite{DBLP:series/isc/DaemenR20} and {\sc Keccak}-$f$ (SHA-3) \cite{DBLP:journals/iacr/BertoniDPA15}.
Some notable examples of asymmetric cryptographic schemes using group algebras include some the post-quantum schemes like {\sc LEDAcrypt} \cite{baldi2019ledacrypt} and {\sc SABER} \cite{DBLP:journals/iacr/DAnversKRV18}.

A ring $R$ is said to be semisimple if $R$ is isomorphic to a direct sum of finitely many simple rings, which are rings with no proper left- or right ideals.
In the commutative case, this is equivalent to saying that $k[G]$ is isomorphic to a direct sum of finitely many fields, since all simple commutative rings are fields.
These simple rings are called the simple components of $R$, and these are unique up to isomorphism.
The \textbf{Wedderburn decomposition} of a semisimple ring $R$ refers to the direct sum into their simple components.
In this paper, we provide a method of determining the Wedderburn decomposition of semisimple group algebras over finite abelian groups.
Instead of using character theory, which is often used in the study of group algebras, we opt for a more geometric approach.
This translates into studying a specific class of coordinate rings, allowing us to use techniques from algebraic geometry and Galois theory.
A major advantage of this approach is that it applies to all fields, whereas character theory often assumes the underlying field to be algebraically closed.

\subsection{Outline}
In Section \ref{SECTION:circulantrings}, we introduce circulant rings which are the geometric equivalence of group algebras over finite abelian groups. 

In Section \ref{SECTION:galoisgroupactions}, we introduce the classical-, algebraic- and geometric Galois group actions, and their role in constructing the Wedderburn decomposition of semisimple circulant rings.

In Section \ref{SECTION:orbit}, we study the the orbit structure of the geometric group action introduced in Section \ref{SECTION:galoisgroupactions} in the special case when the base field is a finite field, which provides us with the Wedderburn decomposition of semisimple commutative group algebras over finite fields.

\subsection{Notation}
\label{SUBSECTION:Notation}

\paragraph{\textbf{Sets and indexing}}
The cardinality or size of a set $S$ is denoted as $\# S$.
The set of all positive integers strictly greater than $0$ is denoted as $\Z_{>0}$.
Similarly we define $\Z_{\geq 0}$.
If $S$ is a finite index set, we refer to $R^{\oplus S}$ as the direct sum of copies of the ring $R$ indexed by $S$.

\paragraph{\textbf{Rings and field extensions}}
Given a commutative ring $R$ with unity, we denote the multiplicative group of invertible elements of $R$ by $R^*$.

For a field $k$, we denote its algebraic closure by $\ovl{k}$. 
Given a field extension $l / k$, we denote the degree of the extension as $[l : k]$.
Moreover, given a tuple $\bx := ( x_1, \ldots ,x_n ) \in \ovl{k}^n$, the smallest field extension over $k$ containing all $x_i$ is denoted by $k(\bx)$.
We denote $\mu_m$ as the set of all $m$-th roots of unity in $\ovl{k}$.
When $l/k$ is a Galois extension, we  write the Galois group as $\Gal(l/k)$.

\paragraph{\textbf{Multivariate polynomials and vanishing set}}
We denote $k[X_1,...,X_n]$ by $A^n_k$.
The set of monomials is referred to as $\Mon(n)$.
Given a polynomial $f \in A^n_k$, we denote the coefficient of $f$ at the monomial $M \in \Mon(n)$ as $\coeff_M(f)$.   
For a monomial $M := \prod_{i=1}^n  X_i^{m_i} \in \Mon(n)$, the total degree of $M$ is defined as $\deg(M) := \sum_{i=1}^n m_i$. 
The $i$-th partial degree of $M$ is defined as $\deg_i(M) := m_i$.
We extend the notions of partial- and total degrees to general polynomials.
for a polynomial $f \in A^n_k$, we define the total degree of $f$ as: 
\begin{align*}
\deg(f) := \max(\deg(M) : M \in \Mon(m) \text{ and } \coeff_M(f) \neq 0).
\end{align*}
Similarly, the $i$-th partial degree of $f$ is defined as: 
\begin{align*}
\deg_i(f) := \max(\deg_i(M) : M \in \Mon(m) \text{ and } \coeff_M(f) \neq 0).
\end{align*}

\begin{example}
Take $f(X_1, X_2, X_3) := X_1^2 - 6 X_1 X_2^2 X_3^3 - 4 X_1 X_2^4 - 7 X_1^2 \in A^3_{\R}$.
Then $\deg(f) = 6$, $\deg_1(f) = 2$, $\deg_2(f) = 4$ and $\deg_3(f) = 3$.
\end{example}

For an ideal $\ad$ of $A^n_k$, we define the vanishing set of $\ad$ as
\begin{align*}
\V(\ad) := \left\{ \bx \in \ovl{k}^n \mid f(\bx) = 0 \text{ for all } f \in \ad \right\}.
\end{align*}


\paragraph{\textbf{Number theory}}
For an integer $m > 0$, we denote the ring of integers modulo $m$ as $\Z_m$.
This is not to be confused with the $p$-adic ring of integers which also uses the same notation in the case when $m = p$.
When referring to an element $x \in \Z_m$, we refer to the integer representation of the residue class of $x$, hence $x \in \{ 0, 1, \ldots, m-1 \}$.

\section{Circulant rings and group algebras}
\label{SECTION:circulantrings}

In this section, we introduce the notion of circulant rings, which are a certain class of coordinate rings.
We also show how circulant rings are related to group algebras over finite abelian groups, and why we can consider circulant rings as the geometric interpretation of such group algebras.

\subsection{Introducing circulant rings}

We provide a formal definition of circulant rings.

\begin{definition}
Let $k$ be a field, and let $\vm := (m_1, \ldots ,m_n) \in \Z_{>0}^n$ be an $n$-tuple of positive integers.
The \textbf{cyclotomic ideal over $k$ with parameters $\vm$} is defined as the ideal $( X_1^{m_1} - 1, \ldots , X_n^{m_n} - 1 )$ in $A^n_k$, and is denoted $\ad_{\vm} / k$, or simply $\ad_{\vm}$ if $k$ is clear from the context.

The corresponding \textbf{circulant ring} is defined as the coordinate ring $A^n_{k} / \ad_{\vm}$,
and we denote this by $\CR_{\vm / k}$.
\end{definition}

\begin{proposition}
\label{PROP:representatives}
The set 
\begin{align*}
\{ f \in A^n_{k} \mid \deg_i(f) < m_i \text{ for all } 1 \leq i \leq n \},
\end{align*}
is a set of representatives of $\CR_{\vm / k}$, which we call the \textbf{standard set of representatives}.
\end{proposition}

\begin{proof}
Apply long division on the partial polynomial degrees.
\end{proof}

\subsection{Group algebras and circulant rings}

\noindent Circulant rings are closely related to group algebras of finite abelian groups.

\begin{theorem}
\label{THM:group_alg-circulant_ring}
Consider the group $G = \Z_{m_1} \times \ldots \times \Z_{m_n}$ where the $m_i$ are positive integers.
Then the following map is an isomorphism of $k$-algebras:
\begin{align*}
\Phi_G \colon k[G] \to \CR_{\vm / k}, \ \sum_{g \in G}  f_g \cdot g \mapsto \sum_{g \in G} f_g \cdot \prod_{i=1}^n  X_i^{g_i},
\end{align*}
where $g := (g_1, \ldots, g_n) \in G$ and $\vm = (m_1, \ldots ,m_n)$.
\end{theorem}

\begin{proof}
The set of polynomials
\begin{align*}
\left\{
f_g \cdot \prod_{i=1}^n X_i^{g_i} : 0 \leq g_i < m_i \text{ and } f_g \in k
\right\},
\end{align*}
is simply another representation of the set of standard representatives introduced in Proposition \ref{PROP:representatives}.
As such, $\Phi_G$ is a well-defined bijective map.
The identities $\Phi_G(f + f') = \Phi_G(f) + \Phi_G(f')$ and $\Phi_G(c \cdot f) = c \cdot \Phi_G(f)$ for $f,f' \in k[G]$ and $c \in k$ are immediate.

For multiplication, define the set $\hat{G} := \{ (g_1, \ldots, g_n) \in \Z_{\geq 0}^n : 0 \leq g_i < m_i \}$, which is the underlying set of $G$ viewed as a subset of the $n$-tuples of integers $\geq 0$.
Observe that:
\begin{align*}
\Phi_G(f) \cdot \Phi_G(f')
&= \left(  \sum_{g \in \hat{G}} f_g \cdot \prod_{i=1}^n  X_i^{g_i}  \right) \cdot \left(  \sum_{g' \in \hat{G}} f'_{g'} \cdot \prod_{i=1}^n  X_i^{g'_i}  \right)  \\
&=  \sum_{\gamma \in \Z_{\geq 0}^n} \left( \sum_{g + g' = \gamma} f_g f'_{g'} \right) \prod_{i=1}^n  X_i^{\gamma}.
\end{align*}
The standard representative in $\CR_{\vm / k}$ of the latter expression is the polynomial
\begin{align*}
\sum_{\gamma \in \hat{G}} \left( \sum_{g + g' \equiv \gamma \bmod \vm} f_g f'_{g'} \right) \prod_{i=1}^n  X_i^{\gamma},
\end{align*}
where $g + g' \equiv \gamma \bmod \vm$ means that $g_i + g'_i \equiv \gamma_i \bmod m_i$ for all $1 \leq i \leq n$.

On the other hand:
\begin{align*}
\Phi_G(f \cdot f')
&= \Phi_G \left( \sum_{\gamma \in G} \left(  \sum_{g + g' = \gamma} f(g) f'(g')\right) \gamma \right) \\
&= \sum_{\gamma \in \hat{G}} \left( \sum_{g + g' \equiv \gamma \bmod \vm} f_g f'_{g'} \right) \prod_{i=1}^n  X_i^{\gamma},
\end{align*}
which shows that $\Phi_G(f) \cdot \Phi_G(f') = \Phi_G(f \cdot f')$.
\end{proof}

The above theorem applies to all finite abelian groups, due to the fundamental theorem of finite abelian groups.

\begin{theorem}[\textbf{Fundamental theorem of finite abelian groups {\cite{lang2004algebra}}}]
\label{Thm-FundThmGroup}
Let $G$ be a finite abelian group.
Then there exists a unique tuple of integers $(m_1, \ldots ,m_n)$ such that $m_i \mid m_{i+1}$ for all $1 \leq i \leq n-1$, and 
\begin{align*}
G \cong \Z_{m_1} \times \ldots \times \Z_{m_n}.
\end{align*}
\end{theorem}

\section{Wedderburn decomposition of semisimple circulant rings}
\label{SECTION:galoisgroupactions}

A consequence of Theorem \ref{THM:group_alg-circulant_ring} is that finding the simple components of semisimple commutative group algebras over $k$ is equivalent to finding simple components of semisimple circulant rings over $k$.
There is a natural analogue of Maschke's theorem for circulant rings, which provides explicit criteria for determining when a circulant ring is semisimple.

\begin{lemma}[\textbf{Maschke's theorem for circulant rings}]
\label{LEMMA:semisimple_condition}
Let $k$ be a field of characteristic $p$, and consider the circulant ring $\CR_{\vm / k}$ where $\vm = (m_1, \ldots, m_n)$.
Then $\CR_{\vm / k}$ is semisimple if and only if $m_1, \ldots, m_n$ are all coprime to $p$.
\end{lemma}

\begin{proof}
By Theorem \ref{THM:group_alg-circulant_ring}, $\CR_{\vm / k}$ is isomorphic to the group algebra $k[G]$, where $G \cong \Z_{m_1} \times \ldots \times \Z_{m_n}$.
Maschke's theorem \cite{james2001representations} states that $k[G]$ is semisimple if and only if the order of $G$ is coprime to $p$.
Observe that $\# G = \prod_{i=1}^n m_i$, which is coprime to $p$ if and only if $m_1, \ldots, m_n$ are all coprime to $p$.
\end{proof}

In this section, we show how one can find the simple components of the Wedderburn decomposition of semisimple circulant rings using techniques from cyclotomic field extensions and algebraic geometry.
Since only semisimple circulant rings are considered, we assume for the remainder of this section that for a field $k$ and for $\vm := (m_1, \ldots, m_n)$ whose entries are all coprime to the characteristic of $k$.

\subsection{Combinatorial Nullstellensatz}

A fundamental result in studying the algebraic structure of coordinate rings in algebraic geometry is Hilbert's Nullstellensatz.
It states a bijective relation between varieties (subsets of $k^n$ which are zeroes of a family of polynomials in $A^n_k$) and the radical ideals of $A^n_k$.
A main downside of this theorem is that it assumes the field $k$ to be algebraically closed.
A weaker version is the Combinatorial Nullstellensatz, where the assumptions for the basefield $k$ are not as strong.

\begin{theorem}[\textbf{Combinatorial Nullstellensatz {\cite{alon2001combinatorial}}}] 
\label{THM:CombNullSatz}
Let $k$ be an arbitrary field, and let $f = f(X_1, . . . , X_n)$ be a polynomial in $A^n_{k}$.
Let $S_1,  \ldots  , S_n$ be nonempty subsets of $k$ and define $g_i(X_i) = \prod_{s \in S_i} (X_i - s)$.
If $f$ vanishes over all the common zeroes of $g_1, \ldots , g_n$ (that is; if $f(s_1, . . . , s_n) = 0$ for all $s_i \in S_i$), then there are polynomials
$h_1, . . . , h_n \in A^n_{k}$ satisfying $\deg(h_i) \leq \deg(f) - \deg(g_i)$ so that
\begin{align*}
f = \sum_{i=1}^n h_i g_i.
\end{align*}
\end{theorem}

We show how the Combinatorial Nullstellensatz is used for constructing the Wedderburn decomposition of semisimple circulant rings, given some restrictions on the basefield $k$.

\begin{lemma}
Let $k$ be any field of characteristic $p$, and let $\vm := (m_1, \ldots ,m_n)$ whose entries are coprime to $p$.
Then 
\begin{align*}
\V(\ad_{\vm}) = \mu_{m_1} \times \ldots \times \mu_{m_n}.
\end{align*}
\end{lemma}

\begin{proof}
This is immediate.
\end{proof}

\begin{theorem}
\label{THM:decom_prelude}
Let $k$ be a field, and consider the semisimple circulant ring $\CR_{\vm / k}$, where $\vm = (m_1, \ldots, m_n)$.
Assume moreover that $\mu_{m_i} \subset k$ for all $1 \leq i \leq n$.
Then we have a well-defined ring isomorphism
\begin{align}
\label{EQ:Wedderburn_embedding}
\tau_{{\vm} / k} \colon \CR_{\vm / k} \to k^{\oplus \V(\ad_{\vm})}, \ f \mapsto (f(\bx))_{\bx \in \V(\ad_{\vm})},
\end{align}
where $k^{\oplus \V(\ad_{\vm})}$ is defined as the direct sum of $\# \V(\ad_{\vm})$ copies of $k$, indexed by $\V(\ad_{\vm})$.
\end{theorem}

\begin{proof}
Let $f$
Consider the ring homomorphism
\begin{align*}
\tau'_{\vm / k} \colon A^n_k \to k^{\oplus \V(\ad_{\vm})}, \ f \mapsto (f(\bx))_{\bx \in \V(\ad_{\vm})}.
\end{align*}
This is a well-defined map since $f(\bx) \in k$ whenever $\bx \in \V(\ad_{\vm})$, and since $\mu_{m_i} \subset k$ for all $1 \leq i \leq n$.

The kernel of $\tau'_{\vm / k}$ consists of exactly the polynomials $f \in A^n_k$ which vanish over $\V(\ad_{\vm})$, which translates to $f(s_1, \ldots, s_n) = 0$ if and only if $s_i \in \mu_{m_i}$ for all $1 \leq i \leq n$.
Since $m_i$ is coprime to the characteristic of $k$, we have the identity $X_i^{m_i} - 1 = \prod_{s_i \in \mu_{m_i}}  (X_i - s_i)$.
Hence by the Combinatorial Nullstellensatz (Theorem \ref{THM:CombNullSatz}), a polynomial $f \in A^n_k$ vanishes over $\V(\ad_{\vm})$ if and only if $f \in (X_1^{m_1} - 1, \ldots, X_n^{m_n} - 1)$, which is by defnition the ideal $\ad_{\vm}$.
As such, $\ker(\tau'_{\vm / k}) = \ad_{\vm}$, which induces the injective homomorphism 
\begin{align*}
A^n_k / \ker(\tau'_{\vm / k}) := \CR_{\vm / k} \to k^{\oplus \V(\ad_{\vm})}.
\end{align*}
This is exactly the map $\tau_{\vm / k}$ in Eq. (\ref{EQ:Wedderburn_embedding}), which proves well-definedness and injectivity.

We only need to show surjectivity.
Looking at the dimension viewed as vector spaces over $k$, we have 
\begin{align*}
\dim_{k}(\CR_{\vm / k}) = \# \V(\ad_{\vm}) = \dim_k \left( k^{\oplus \V(\ad_{\vm})} \right).
\end{align*}
Since $\tau_{\vm / k}$ is an injective linear map, it must thus also be surjective due to the dimensions of the vector spaces being equal.
\end{proof}

\subsection{Galois group actions}

Theorem \ref{THM:decom_prelude} presents the simple components of semisimple circulant rings, given that their base field $k$ is ``large enough''.
We refine this decomposition to any field $k$, which requires some results from Galois theory and cyclotomic extensions.

\begin{definition}
\label{DEF:k_m}
Let $k$ be a field of characteristic $p$, and consider an $n$-tuple $\vm = (m_1, \ldots, m_n) \in \Z_{>0}$ whose entries are coprime to $p$.
Then $k_{\vm}$ is defined as the cyclotomic extension $k(\mu_{m_1}, \ldots, \mu_{m_n})$, which is the smallest field extension of $k$ containing $\mu_{m_i}$ for all $1 \leq i \leq n$.
\end{definition}

Since $k_{\vm} / k$ is a cyclotomic extension, it is a Galois extension with abelian Galois group $\Gal(k_{\vm} / k)$.
In this section, we discuss three group actions by $\Gal(k_{\vm} / k)$ which are key to constructing the semisimple decomposition of circulant rings.

\begin{definition}[\textbf{Classical group action}]
We define the classical group action as
\begin{align*}
\Gal(k_{\vm} / k) \times k_{\vm} \to k_{\vm}, \ (\sigma, x) \mapsto \sigma(x).
\end{align*}
\end{definition}

\begin{definition}[\textbf{Algebraic group action}]
For any $\sigma \in \Gal(k_{\vm} / k)$ and for any $f \in \CR_{\vm / k_{\vm}}$, we define $\sigma(f) := \sum_{M \in \Mon(n)} \sigma(\coeff_M(f)) \cdot M$.
This induces the group action
\begin{align*}
\Gal(k_{\vm} / k) \times \CR_{\vm / k_{\vm}}  \to \CR_{\vm / k_{\vm}}, \ (\sigma, f) \mapsto \sigma(f),
\end{align*}
which we call the \textbf{algebraic group action}.
\end{definition}

\begin{remark}
\label{REMARK:Gal-invariant}
The $\Gal(k_{\vm} / k)$-invariants of $\CR_{\vm / k_{\vm}}$ equals the set $\CR_{\vm / k}$.
To see this, note that for $f \in \CR_{\vm / k_{\vm}}$, $\sigma(f)$ only affects the coefficients of the terms of $f$.
As such by Proposition \ref{PROP:representatives}, given that $f$ is contained in the standard set of representatives, so is $\sigma(f)$.
Hence $f = \sigma(f)$ in $\CR_{\vm / k_{\vm}}$ if and only if the coefficients of $f$ stays invariant under $\sigma$.
By the fundamental theorem of Galois theory \cite[Theorem~4.10.1]{ehrlich2011fundamental}, this happens if and only if all coefficients of $f$ are contained in $k$, or equivalently when $f \in \CR_{\vm / k}$. 
\end{remark}

We now introduce the group action whose structure plays a key role in the decomposition of semisimple circulant rings.

\begin{definition}[\textbf{Geometric group action}]
\label{DEF:geom_alpha}
For any $\sigma \in \Gal(k_{\vm} / k)$ and for any $\bx := (x_1, \ldots ,x_n) \in \V(\ad_{\vm})$, we define $\sigma(\bx) := (\sigma(x_1), \ldots ,\sigma(x_n))$.
This induces the group action
\begin{align*}
\alpha_{\vm / k} \colon \Gal(k_{\vm} / k) \times \V(\ad_{\vm}) \to \V(\ad_{\vm}), \ (\sigma, \bx) \mapsto \sigma(\bx),
\end{align*}
which we call the \textbf{geometric group action}.

For $\bx \in \V(\ad_{\vm})$, we denote $\Orb(\bx)$ as the corresponding orbit of $\bx$ under $\alpha_{\vm / k}$.
\end{definition}

The classical-, algebraic-, and the geometric group actions are related in the following way:

\begin{lemma}
\label{LEM:actionID}
For any $\sigma \in \Gal(k_{\vm} / k)$, $\bx \in \V(\ad_{\vm})$ and $f \in\CR_{\vm / k_{\vm}}$, we have the identity
\begin{align*}
\sigma(f(\bx)) = \sigma(f) (\sigma(\bx)).
\end{align*}
\end{lemma}

\begin{proof}
Since $\sigma$ as an automorphism splits under addition and multiplication of elements in $k_{\vm}$, we have that
\begin{align*}
\sigma(f(\bx))
&= \sigma \left( \sum_{M \in \Mon(n)} \coeff_M(f) \cdot M(\bx) \right) \\
&= \sum_{M \in \Mon(n)} \sigma(\coeff_M(f)) \cdot \sigma(M(\bx)) \\
&= \sum_{M \in \Mon(n)} \sigma(\coeff_M(f)) \cdot M(\sigma(\bx)) \\
&= \sigma(f)(\sigma(\bx)),
\end{align*}
which concludes the proof.
\end{proof}

\begin{corollary}
\label{COR:invar_f}
Let $f \in \CR_{\vm / k_{\vm}}$ with all coefficients in $k$. 
Then for all $\sigma \in \Gal(k_{\vm} / k)$, the identity $\sigma(f(\bx)) = f (\sigma(\bx))$ holds.
\end{corollary}

\begin{lemma}
Let $\by \in \V(\ad_{\vm})$ and let $f \in \CR_{\vm / k}$.
Then 
\begin{align*}
\prod_{\bx \in \Orb(\by)} f(\bx) \in k.
\end{align*}
\end{lemma}

\begin{proof}
Observe that any automorphism $\sigma \in \Gal(k_{\vm} / k)$ induces a natural bijection $\Orb(\by) \to \Orb(\by), \ \bx \mapsto \sigma(\bx)$.
Thus we have  
\begin{align*}
\sigma \left( \prod_{\bx \in \Orb(\by)} f(\bx) \right)
= \prod_{\bx \in \Orb(\by)} f(\sigma(\bx))
= \prod_{\sigma(\bx) \in \Orb(\by)} f(\bx)
= \prod_{\bx \in \Orb(\by)} f(\bx).
\end{align*}
From the fundamental theorem of Galois theory, we conclude that the expression $\prod_{\bx \in \Orb(\by)} f(\bx)$ is indeed contained in $k$.
\end{proof}

\begin{lemma}
For $\bx \in \V(\ad_{\vm})$, we have $[k(\bx) : k] = \# \Orb(\bx)$.
\end{lemma}

\begin{proof}
The field extension $k(\bx) / k$ is a cyclotomic extension, which is always a Galois extension.
The result now follows directly from the fundamental theorem of Galois theory.
\end{proof}

\subsection{Wedderburn decomposition}

We present the full Wedderburn decomposition of semisimple cyclotomic rings.

\begin{theorem}
\label{THM:delta}
Let $\by \in \V(\ad_{\vm})$, then there exists a polynomial $\delta_{\Orb(\by)} \in \CR_{\vm / k}$ such that
\begin{align*}
\delta_{\Orb(\by)}(\bx) = 
\begin{cases}
1 \text{ if } \bx \in \Orb(\by) \\
0 \text{ else}.
\end{cases}
\end{align*}
\end{theorem}

\begin{proof}
By Theorem \ref{THM:decom_prelude} we have the isomorphism
\begin{align*}
\tau_{\vm / k_{\vm}} \colon \CR_{\vm / k_{\vm}} \to k_{\vm}^{\oplus \V(\ad_{\vm})}, \ f \mapsto (f(\bx))_{\bx \in \V(\ad_{\vm})}.
\end{align*}
As such, there exists a polynomial $g \in \CR_{\vm / k_{\vm}}$ satisfying 
\begin{align*}
g(\bx) = 
\begin{cases}
1 \text{ if } \bx \in \Orb(\by) \\
0 \text{ else}.
\end{cases}
\end{align*}
We will show that $g \in \CR_{\vm / k}$.

Assume to the contrary that $g \notin \CR_{\vm / k}$.
Then there exists a coefficient $c$ of $g$ such that $c \notin k$.
Remark \ref{REMARK:Gal-invariant} then implies that there exists a $\sigma \in \Gal(k_{\vm} / k)$ such that $\sigma(g) \neq g$ in $\CR_{\vm / k_{\vm}}$.
The isomorphism $\tau_{\vm / k_{\vm}}$ implies that there exists $\bx \in \V(\ad_{\vm})$ such that 
\begin{align}
\label{EQ:sigma(g)_g}
\sigma(g)(\bx) \neq g(\bx).
\end{align}
By Lemma \ref{LEM:actionID}, we have that $\sigma(g(\bx)) = \sigma(g) (\sigma(\bx))$.
Since $\bx$ and $\sigma(\bx)$ are in the same orbit, we get $\sigma(g)(\sigma(\bx)) = \sigma(g)(\bx)$ because $g$, and thus also $\sigma(g)$, is constant over orbits with values in $k$.
This results into the equation $\sigma(g(\bx)) =  \sigma(g)(\bx)$.
Note however that since $g(\bx) \in k$, we have $\sigma(g(\bx)) = g(\bx)$, which results into the equation $\sigma(g)(\bx) = g(\bx)$.
This is a contradiction to Eq. (\ref{EQ:sigma(g)_g}).
Hence all coefficients of $g$ are contained in $k$, which means that $g \in f \in \CR_{\vm / k}$.
Choosing $\delta_{\Orb(\by)} = g$ concludes the proof.
\end{proof}

\begin{lemma}
\label{LEM:exist_poly}
Let $\bx \in \V(\ad_{\vm})$ and let $a \in k(\bx)$.
Then there exists a polynomial $p_{a, \bx} \in \CR_{\vm / k}$ such that $p_{a, \bx} (\bx) = a$.
\end{lemma}

\begin{proof}
Let us use the expression $\bx = (x_1, \ldots ,x_n)$.
We proceed by induction on $n$.
For $n = 1$, we have by basic field theory that there exists unique elements $c_0,c_1, \ldots ,c_t$ in $k$ with $t < [k(x_1) : k]$ such that 
\begin{align*}
a = c_0 + c_1 \cdot x_1 + c_2 \cdot x_1^2 + \ldots + c_t \cdot x_1^t.
\end{align*}
Hence the polynomial 
\begin{align*}
p_{a, x_1} (X_1) := \sum_{i=0}^t c_i \cdot X_1^i,
\end{align*}
satisfies the assumption, thus proving the case for $n = 1$.

Now assume the lemma is true for $n = j$ for some integer $j > 1$, and consider $n = j+1$.
Observe that $k(x_1, \ldots x_j, x_{j+1}) = k(x_1, \ldots x_j)(x_{j+1})$.
With a similar argument as in the case for $n = 1$, there exists unique elements $c_0,c_1, \ldots ,c_t$ in $k(x_1, \ldots ,x_j)$ with $t < [k(x_1, \ldots, x_j)(x_{j+1}) : k(x_1, \ldots, x_k)]$ such that 
\begin{align*}
a = c_0 + c_1 \cdot x_{j+1} + c_2 \cdot x_{j+1}^2 + \ldots + c_t \cdot x_{j+1}^t.
\end{align*}
By the induction hypotheses, there exist polynomials $p_0, \ldots ,p_t \in k[X_1, \ldots ,X_j]$ such that $p_i(x_1, \ldots ,x_j) = c_i$ for all $0 \leq i \leq t$.
Hence the polynomial
\begin{align*}
p_{a, \bx} = \sum_{i=0}^t p_i(X_1, \ldots ,X_j) \cdot X_{j+1}^t,
\end{align*}
satisfies our assumptions, thus concluding the proof.
\end{proof}

\begin{theorem}
\label{THM:polyconstruct}
Let $S \subset \V(\ad_{\vm})$ be a set of representatives of the orbits of $\alpha_{\vm / k}$.
Assume that we have a map $F : S \to \ovl{k}$ such that $F(\bx) \in k(\bx)$ for all $\bx \in S$.
Then there exists a polynomial $f \in A^n_k$ such that $f(\bx) = F(\bx)$.
\end{theorem}

\begin{proof}
Consider the polynomial
\begin{align*}
f =  \sum_{\by \in S} p_{F(\by), \by} \cdot \delta_{\Orb(\by)},
\end{align*}
with $p_{F(\by),\by}$ as in Lemma \ref{LEM:exist_poly} and $\delta_{\Orb(\by)}$ as in Theorem \ref{THM:delta}.
Then for any $\bx \in S$, we have
\begin{align*}
f(\bx) 
= \sum_{\by \in S} p_{F(\by) , \by}(\bx) \cdot \delta_{\Orb(\by)}(\bx)
= p_{F(\bx), \bx}(\bx) \cdot \delta_{\Orb(\bx)}(\bx)
= F(\bx) \cdot 1
= F(\bx),
\end{align*}
where the second equality is because $\delta_{\Orb(\by)} (\bx) = 0$ for $\by \neq \bx$.
This concludes the proof.
\end{proof}

\begin{theorem}[\textbf{Wedderburn decomposition}]
\label{THM:GenDec}
Let $k$ be a field, and let $\CR_{\vm / k}$ be semisimple.
Let $S \subset \V(\ad_{\vm})$ be a set of representatives of the orbits of $\alpha_{\vm / k}$.
Then we have the isomorphism
\begin{align}
\label{EQ:GenDec}
\CR_{\vm / k} \to \bigoplus_{\bx \in S} k(\bx), \ f \mapsto (f(\bx))_{\bx \in S},
\end{align}
where $[k(\bx) : k] = \# \Orb(\bx)$.
\end{theorem}

\begin{proof}
Since $\tau_{\vm / k_{\vm}}$ from Theorem \ref{THM:decom_prelude} is an isomorphism, $\tau_{\vm / k_{\vm}} \mid_{\CR_{\vm / k}}$ (restricted to $\CR_{\vm / k}$) is an injective homomorphism.
Using Corollary \ref{COR:invar_f}, we can refine $\tau_{\vm / k_{\vm}} \mid_{\CR_{\vm / k}}$  to the injective homomorphism
\begin{align}
\label{EQ:GenDec1}
\CR_{\vm / k} \to k_{\vm}^{\oplus S}, \ f \mapsto (f(\by))_{\by \in S}.
\end{align}
Clearly, $f(\by) \in k(\by)$ for all $\by \in S$ and $f \in A^n_k$.
If given elements $a_{\by} \in k(\by)$ for each $\by \in S$, we conclude from Theorem \ref{THM:polyconstruct} that there exists a polynomial $f \in A^n_k$ such that $f(\by) = a_{\by}$.
Hence (\ref{EQ:GenDec1}) induces the desired isomorphism (\ref{EQ:GenDec}).
\end{proof}

\begin{remark}
The simple components in de decomposition are isomorphic to cyclotomic field extensions of $k$.
This means in particular that $\CR_{\vm / k}$ is an \'etale algebra over $k$ if it is semisimple. 
\end{remark}

\section{Wedderburn decomposition over finite fields}
\label{SECTION:orbit}

In this section, we present the Wedderburn decomposition over semisimple circulant rings over finite fields.
The decomposition over finite fields is particularly well-behaved, due to the good understanding of field extensions of finite fields.
For the remainder of this section, we assume by default that for a finite field $\F_q$ and $\vm := (m_1, \ldots, m_n)$ whose entries are all coprime to $q$, or equivalently when $\CR_{\vm / \F_q}$ is semisimple.

\subsection{Some additional notation}

We introduce some additional notation, building on $\Z_m$ defined under Number theory in Section \ref{SUBSECTION:Notation}.
We denote the multiplicative group of $\Z_m$ by $\Z_m^*$.
The size of $\Z_m^*$ is denoted $\varphi(m)$, which is known as Euler's totient function.
For $a \in \Z_m^*$, we denote the multiplicative order of $a$ by $\ord_m(a)$.
We define $\Div_m$ as the set of all integers dividing $m$, including $1$ and $m$ itself.

Let $\vm = (m_1, \ldots, m_n)$ be an $n$-tuple of positive integers.
We define the product ring $\Z_{\vm} := \Z_{m_1} \times \ldots \times \Z_{m_n}$.
As such, the multiplicative group of $\Z_{\vm}$, which we denote $\Z^*_{\vm}$, equals $\Z^*_{m_1} \times \ldots \times \Z^*_{m_n}$.
From elementary group theory, the order of $\Z^*_{\vm}$ equals $\prod_{i=1}^n \varphi(m_i)$.
Lastly, we define $\Div_{\vm}$ as the direct product $\Div_{m_1} \times \ldots \times \Div_{m_n}$.

\subsection{Some modular arithmetic}

\begin{definition}
\label{DEF:multivar}
For $\vd = (d_1, \ldots, d_n) \in \Div_{\vm}$, we define the set 
\begin{align*}
\Z_{\vm}(\vd) := \left\{ (x_1, \ldots, x_n) \in \Z_{\vm} : \gcd(x_i, m_i) = d_i \text{ for all } 1 \leq i \leq n \right\},
\end{align*}
\end{definition}

\begin{proposition}
\label{PROP:disjoint_covering}
We have the disjoint union $\Z_{\vm} = \bigcupdot_{\vd \in \Div_{\vm}} \Z_{\vm}(\vd)$.
\end{proposition}

\begin{proof}
First we show that the subsets $\Z_{\vm}(\vd)$ where $\vd \in \Div_{\vm}$ cover $\Z_{\vm}$.
Let $\bx = (x_1, \dots, x_n) \in \Z_{\vm}$, and let $\vd^* := (d^*_1, \ldots, d^*_n)$ where $d^*_i := \gcd(x_i, m_i)$.
Then by definition, $\bx \in \Z_{\vm}(d^*)$.
Since this is true for all $\bx \in \Z_{\vm}$, the subsets $\Z_{\vm}(d)$ indeed cover $\Z_{\vm}$.

We show that $\Z_{\vm}(\vd)$ and $\Z_{\vm}(\vd')$ are disjoint for $\vd \neq \vd'$.
Assume to the contrary that these sets are not disjoint.
Then there exists an $\bx = (x_1, \ldots, x_n) \in \Z_{\vm}$ contained in both sets.
Since $\vd \neq \vec{d'}$, there exists an $1 \leq i \leq n$ such that $d_i \neq d'_i$.
The assumption $x \in \Z_{\vm}(\vd)$ implies $\gcd(x_i, m_i) = d_i$, and similarly, $x \in \Z_{\vm}(\vd')$ implies $\gcd(x_i, m_i) = d'_i$.
But then $d_i = d'_i$, which is a contradiction.
Hence $\Z_{\vm}(\vd)$ and $\Z_{\vm}(\vd')$ must be disjoint.
\end{proof}

\begin{lemma}
\label{LEMMA:invertible_elements_bijection}
There is a natural bijection $\omega \colon \Z^*_{\vm / \vd} \to \Z_{\vm} (\vd), \ \vec{c} \mapsto \vd \cdot \vec{c}$, where $\vm / \vd := (m_1 / d_1, \ldots, m_n / d_n)$.
\end{lemma}

\begin{proof}
The set of $n$-tuples of positive integers
\begin{align*}
\left\{ (d_1 \cdot c_1, \ldots, d_n \cdot c_n) \in \Z_{>0}^n : 1 \leq c_i \leq m_i / d_i \text{ and } \gcd(c_i,d_i) = 1  \right\},
\end{align*}
contains exactly the $n$-tuple integer representatives of all elements in $\Z_{\vm}$ in the subset $\Z_{\vm}(\vd)$.
Observe that the elements in the set 
\begin{align*}
\{ (c_1, \ldots, c_n) \in \Z_{>0}^n : 1 \leq c_i \leq m_i / d_i \text{ and } \gcd(c_i,d_i) = 1  \},
\end{align*}
are exactly the $n$-tuple integer representatives of $\Z_{\vm / \vd}^*$, which proves the claim.
\end{proof}

\begin{remark}
\label{REMARK:univar}
We state some trivial but useful facts:
\begin{itemize}
\item[1.]  	$\Z_{\vm}(\vm) = \{ \vec{0} \}$ and $\Z_{\vm}( \vec{1}) = \Z_{\vm}^*$ where $\vec{0} := (0 \bmod m_i)_{1 \leq i \leq n}$ and $\vec{1} := (1 \bmod m_i)_{1 \leq i \leq n}$;
\item[2.] 	By Lemma \ref{LEMMA:invertible_elements_bijection}, we have $\# \Z_{\vm}(\vd) = \# \Z^*_{\vm / \vd} = \prod_{i = 1}^n \varphi(m_i/d_i)$.
\end{itemize}
\end{remark}

\subsection{Modular group action}

We construct a group action which is equivalent to $\alpha_{\vm/ \F_q}$ (Definition \ref{DEF:geom_alpha}), but which is computationally more convenient.

\begin{definition}[\textbf{Modular group action}]
\label{DEF:mod_group_act}
Let $\vm := (m_1, \ldots, m_n)$ be an $n$-tuple of positive integers whose entries are coprime to $q$.
Define the element 
\begin{align*}
\vec{q} :=  (q \bmod m_1, \ldots, q \bmod m_n) \in \Z_{\vm}^*,
\end{align*}
which is well defined since $q$ and $m_i$ are coprime, and let $\langle \vec{q} \rangle$ be the cyclic subgroup of $\Z_{\vm}^*$ generated by $\vec{q}$.
For $\bx := (x_1, \ldots, x_n) \in \Z_{\vm}$ and for all $t \geq 1$, we define the \textbf{modular group action} as:
\begin{align*}
\alpha_{\vm / \F_q}^* \colon \langle \vec{q} \rangle \times \Z_{\vm} \to \Z_{\vm}, \ (\vec{q}^t, \bx) \mapsto (x_1 \cdot q^t \bmod m_1, \ldots, x_n \cdot q^t \bmod m_n).
\end{align*}
If $\vm$ and $\F_q$ are clear from context, we denote $\alpha_{\vm / \F_q}^*(\vec{q}^t, \bx)$ simply by $\bx \cdot \vec{q}^t$.  
\end{definition}

\begin{theorem}
\label{THM:equivalence_group_actions}
The group actions $\alpha_{\vm/k}^*$ and $\alpha_{\vm/k}$ from Definition \ref{DEF:geom_alpha} are equivalent.
\end{theorem}

\begin{proof}
We construct a group isomorphism $\iota \colon \Gal(k_{\vm} / k) \to \langle \vec{q} \rangle$ and a bijective map $\gamma \colon \V(\ad_{\vm}) \to \Z_{\vm}$ such that the following diagram commutes:
$$
\begin{tikzcd}
\Gal(k_{\vm} / k) \times  \V(\ad_{\vm})				\arrow[r, "\alpha_{\vm/k}"] \arrow[d, "\iota \times \gamma"'] 	& \V(\ad_{\vm})  \arrow[d, "\gamma"]  \\
\langle \vec{q} \rangle \times \Z_{\vm} 	\arrow[r, "\alpha_{\vm/k}^*"]								&  \Z_{\vm}
\end{tikzcd}
$$

\paragraph{\textbf{The map $\iota$}}
The Galois group $\Gal(k_{\vm} / k)$ is cyclic of order $\lcm_{i=1}^n (\ord_{m_i}(q))$, and it is generated by the Frobenius automorphism $\sigma_q \colon k_{\vm} \to k_{\vm}, \ x \mapsto x^q$.
Hence there is a natural group isomorphism 
\begin{align*}
\iota \colon \Gal(k_{\vm} / k) \to \langle \vec{q} \rangle, \ \sigma_q^t \mapsto \vec{q}^t.
\end{align*}

\paragraph{\textbf{The map $\gamma$}}
For each $1 \leq i \leq n$, let us fix some primitive $m_i$-th root $\zeta_{m_i}$.
Note that every element in $\mu_{m_i}$ can be uniquely expressed as $\zeta_{m_i}^{a_i}$ where $0 \leq a_i < m_i$.
Depending on the choice of the primitive roots, we can construct a natural bijective map
\begin{align*}
\gamma \colon \V(\ad_{\vm}) \to \Z_{\vm}, \ (\zeta_{m_1}^{x_1}, \ldots ,\zeta_{m_n}^{x_n}) \mapsto (x_1, \ldots ,x_n).
\end{align*}
This map is bijective since $\# \V (\ad_{\vm}) = \# \Z_{\vm}$.

\paragraph{\textbf{Commutativity}}
Confirming commutativity of the diagram is a matter of writing down both compositions, and show that the outputs are equal:
Assume we have some $\sigma_q^t \in \Gal(k_{\vm} / k)$ and $(\zeta_{m_1}^{x_1}, \ldots ,\zeta_{m_n}^{x_n}) \in \V(\ad_{\vm})$.
Then we have
\begin{align*}
\gamma \circ \alpha_{\vm/k}(\sigma_q^t, (\zeta_{m_1}^{x_1}, \ldots ,\zeta_{m_n}^{x_n})) 
&= \gamma \left(  \zeta_{m_1}^{x_1 \cdot q^t}, \ldots ,\zeta_{m_n}^{x_n \cdot q^t}  \right) \\
&= (x_1 \cdot q^t \bmod m_1,  \ldots , x_n \cdot q^t \bmod m_n),
\\
\alpha_{m/k}^* \circ (\iota \times \gamma)  (\sigma_q^t, (\zeta_{m_1}^{x_1}, \ldots ,\zeta_{m_n}^{x_n})) 
&= \alpha_{m/k}^* \left( q^t \bmod m,  (x_i \bmod m_i)_{1 \leq i \leq m} \right) \\
&= (x_1 \cdot q^t \bmod m_1,  \ldots , x_n \cdot q^t \bmod m_n).
\end{align*}
This implies $\gamma \circ \alpha_{\vm/k} = \alpha_{\vm/k}^* \circ (\iota \times \gamma)$, which proves commutativity of the diagram.
\end{proof}

\begin{remark}
The equivalence of $\alpha_{\vm/k}$ and $\alpha_{\vm/k}^*$ is not canonical, as it depends on the choice of primitive roots.
\end{remark}

\subsection{Orbit sctructure}

We discuss the orbit structure of the modular group action $\alpha_{\vm/ \F_q}^*$.

\begin{lemma}
\label{LEMMA:orbit_contained}
Let $\bx \in \Z_{\vm}(\vd)$, then $\Orb(\bx)$ is contained in $\Z_{\vm}(\vd)$.
\end{lemma}

\begin{proof}
Since $q$ is coprime to $m_i$, $\gcd(m_i, x_i \cdot q^t) = \gcd(m_i, x_i) = d_i$ for all $1 \leq i \leq n$ and $t \geq 0$.
As such, $\bx \cdot \vec{q}^t \in \Z_{\vm}(\vd)$, hence $\Orb(\bx) \subseteq \Z_{\vm}(\vd)$.
\end{proof}

From the above result, the group action $\alpha^*_{\vm / \F_q}$ restricted to $\Z_{\vm}(\vd)$ is a well-defined group action.
As such, it induces an equivalence relation on $\Z_{\vm}(\vd)$ where the orbits are the equivalence classes.

Consider the group $\Z^*_{\vm/\vd}$ with the equivalence relation induced from the quotient group $\Z^*_{\vm / \vd} / \langle \vec{q}_d \rangle$ where $\vec{q}_d := (q \bmod m_i / d_i)_{1 \leq i \leq n} \in \Z^*_{\vm / \vd}$.

\begin{lemma}
\label{LEMMA:congruence_relation_preservation}
Consider $\Z^*_{\vm/\vd}$ and $\Z_{\vm} (\vd)$ with both their respective equivalence relations as defined above.
Let $\omega \colon \Z^*_{\vm/\vd} \to \Z_{\vm} (\vd)$ be the bijective map defined in Lemma \ref{LEMMA:invertible_elements_bijection}.
Then for every $\bx \in \Z^*_{\vm / \vd}$, the image of the equivalence class of $\bx$ under $\omega$ is exactly the equivalence class of $\omega(\bx)$.
\end{lemma}

\begin{proof}
By defintion, the equivalence class of $\bx := (x_1, \ldots, x_n)$ in $\Z^*_{\vm/\vd}$ consists of exactly the elements 
\begin{align*}
[\bx] = \left\{ (x_i \cdot q^t \bmod (m_i / d_i))_{1 \leq i \leq n} : t \geq 1  \right\}.
\end{align*}
Observe that 
\begin{align*}
\omega([\bx]) =  \left\{ (d_i \cdot (x_i \cdot q^t \bmod (m_i/d_i)) \bmod m_i)_{1 \leq i \leq n} : t \geq 1  \right\},
\end{align*}
which by definition is exactly the orbit $\Orb(\omega(\bx))$.
\end{proof}

\begin{proposition}
\label{PROP:local_orbit_structure}
Consider the group action $\alpha^*_{\vm / \F_q}$ restricted to $\Z_{\vm}(\vd)$.
The number of orbits in $\Z_{\vm} (\vd)$ equals:
\begin{align*}
\frac{\prod_{i=1}^n \varphi(m_i)}{\lcm_{1 \leq i \leq n} (\ord_{m_i/d_i}(q))};
\end{align*}
Moreover, all orbits in $\Z_{\vm} (\vd)$ have the same length, being:
\begin{align*}
\lcm_{1 \leq i \leq n} (\ord_{m_i/d_i}(q)). 
\end{align*}
\end{proposition}

\begin{proof}
By Lemma \ref{LEMMA:congruence_relation_preservation}, the number of orbits equals $\# \Z^*_{\vm / \vd} / \langle \vec{q}_d \rangle = \frac{\# \Z^*_{\vm / \vd}}{\# \langle \vec{q}_d \rangle}$.
It is immediate that $\# \Z^*_{\vm / \vd} = \prod_{i=1}^n \varphi(m_i)$.
To see the expression of the denominator, observe that the order of $\langle \vec{q}_d \rangle$ is the smallest integer $t$ such that $q^t \equiv 1 \bmod (m_i/d_i)$ for all $1 \leq i \leq n$ simultaneously.
This number equals $\lcm_{1 \leq i \leq n} (\ord_{m_i/d_i}(q))$, which proves the first claim.

Every equivalence class in $\Z^*_{\vm/\vd}$ has the same size, which equals $\#\langle \vec{q}_d \rangle = \lcm_{1 \leq i \leq n} (\ord_{m_i/d_i}(q))$.
This statement translates to the orbits of $\Z_{\vm}(\vd)$ by Lemma \ref{LEMMA:congruence_relation_preservation}.
\end{proof}

\begin{theorem}[\textbf{Orbit structure theorem over finite fields}]
\label{THM:orbit_structure}
Consider the finite field $\F_q$, and let $\vm = (m_1, \ldots, m_n)$ be an $n$-tuple whose entries are coprime to $q$.
Then the following statements hold regarding the orbits of $\alpha_{\vm / \F_q}^*$:
\begin{itemize}
\item[1.]	For every orbit $\Orb(\bx)$ of $\alpha_{\vm / \F_q}^*$, there exists a unique $\vd \in \Div_{\vm}$ such that $\Orb(\bx)$ is contained in $\Z_{\vm}(\vd)$;
\item[2.]	Let $\vd := (d_1, \ldots, d_n) \in \Div_{\vm}$.
Then every orbit in $\Z_{\vm}(\vd)$ is of size $\lcm_{1 \leq i \leq n} (\ord_{m_i/d_i}(q))$, and the number of orbits in $\Z_{\vm}(\vd)$ equals:
\begin{align*}
\frac{\prod_{i=1}^n \varphi(m_i)}{\lcm_{1 \leq i \leq n} (\ord_{m_i/d_i}(q))};
\end{align*}
\item[3.]	The total number of orbits of $\alpha_{\vm / \F_q}^*$ equals:
\begin{align*}
\sum_{(d_1, \ldots, d_n) \in \Div_{\vm} } \frac{\prod_{i=1}^n \varphi(d_i)}{\lcm_{1 \leq i \leq n}(\ord_{m_i/d_i}(q))}.
\end{align*}	
\end{itemize}
\end{theorem}

\begin{proof}
Given some $\bx \in \Z_{\vm}$, there exists a unique $\vd \in \Div_{\vm}$ such that $\bx \in \Z_{\vm}(\vd)$ (see Proposition \ref{PROP:disjoint_covering}).
By Lemma \ref{LEMMA:orbit_contained}, $\Orb(\bx) \subset \Z_{\vm}(\vd)$, which proves the first statement.

The second statement is simply a reformulation of Proposition \ref{PROP:local_orbit_structure}.

The third statement is an immedate consequence of the first and second statement.
\end{proof}

\begin{remark}
In the univariate case ($n=1$), the orbit structure provides a proof of Gauss's identity, which states that $m = \sum_{d \mid m} \varphi(d)$ for $m \in \Z_{>0}$.
Now let $p$ be any prime number coprime to $m$.
The main idea is to express the size of $\Z_m(d)$ (which equals $m$) as the sum of the lenghts of all orbits of $\alpha^*_{m/\F_p}$ in $\Z_m$.
From Theorem \ref{THM:orbit_structure}, this idea translates as follows:
\begin{align*}
m = \# \Z_m = \sum_{d \mid m} \frac{\varphi(m)}{\ord_{m/d}(p)} \cdot \ord_{m/d}(p) = \sum_{d \mid m} \varphi(d),
\end{align*}
which is Gauss's identity.
\end{remark}

\begin{theorem}[\textbf{Wedderburn decomposition over finite fields}]
\label{THM:Wedderburn_finite_field}
Let $\CR_{\vm / \F_q}$ be semisimple.
For any $\vd := (d_1, \ldots, d_n) \in \Div_{\vm}$, define the expressions $\nu_{\vm / \F_q}(\vd) := \lcm_{1 \leq i \leq n} (\ord_{m_i/d_i}(q))$ and $\eta_{\vm / \F_q}(\vd) := \frac{\prod_{i=1}^n \varphi(m_i)}{\nu_{\vm,q}(\vd)}$.
Then the Wedderburn decomposition of $\CR_{\vm / \F_q}$ into simple components equals up to isomorphism:
\begin{align*}
\CR_{\vm / \F_q} \cong \bigoplus_{\vd \in \Div_{\vm}} \left( \F_{q^{\nu_{\vm / \F_q}(\vd)}} \right)^{\eta_{\vm / \F_q}(\vd)}.
\end{align*}
\end{theorem}

\begin{proof}
Let $S$ be a set of representatives of the orbits of the group action $\alpha_{\vm / q}$.
From Theorem \ref{THM:GenDec}, the simple components of $\alpha_{\vm / q}$ are the field extensions $\F_q(\bx)$ for all $\bx \in S$, where we have $[\F_q(\bx) : \F_q] = \# \Orb(\bx)$.
An important result in the theory of finite fields is that every finite field is, up to isomorphism, uniquely determined by its cardinality.
As such, $\F_q(\bx)$ is isomorphic to the field $\F_{q^{\# \Orb(\bx)}}$.
From Theorem \ref{THM:equivalence_group_actions}, $\alpha_{\vm / q}$ is equivalent to $\alpha^*_{\vm / q}$, of which the orbits are well-understood (Theorem \ref{THM:orbit_structure}).
The rest follows from the first and second statement of Theorem \ref{THM:orbit_structure}.
\end{proof}

\subsection{An example}
We demonstrate by means of an example how Theorem \ref{THM:Wedderburn_finite_field} can be used to determine the Wedderburn decomposition of a semisimple circulant ring over a finite field.
Consider the semisimple circulant ring
\begin{align*}
\CR_{(5,5) / \F_2} := \F_2[X_1,X_2] / (X_1^5 - 1, X_2^5 - 1).
\end{align*}
Let us first compute $\Div_{(5,5)}$:
\begin{align*}
\Div_{(5,5)} := \{ (1,1), (1,5), (5,1), (5,5) \}.
\end{align*}
In view of Theorem \ref{THM:Wedderburn_finite_field}, we compute the values of $\prod_{i=1}^2 \varphi(d_i)$, $\nu_{(5,5) / \F_2}(\vd)$ and $\eta_{(5,5) / \F_2}(\vd)$ for each $\vd \in \Div_{(5,5)}$:
\begin{table}[!h]
\begin{center}
	\begin{tabular}{| l | l | l | l |}
	\hline
	$\vd = (d_1, d_2)$	& $\prod_{i=1}^2 \varphi(d_i)$		& $\nu_{(5,5) / \F_2}(\vd)$	& $\eta_{(5,5) / \F_2}(\vd)$ 	\\ \hline
	$(1,1)$				& $1$								& $1$ 						& $1$						\\
	$(1,5)$				& $4$								& $4$						& $1$						\\
	$(5,1)$				& $4$								& $4$						& $1$						\\
	$(5,5)$				& $16$								& $4$						& $4$						\\
	\hline
	\end{tabular}
\caption{Computing $\nu_{(5,5) / \F_2}(\vd)$ and $\eta_{(5,5) / \F_2}(\vd)$}
\end{center}
\end{table}

By Theorem \ref{THM:Wedderburn_finite_field}, the number of indecomposable components of $\CR_{(5,5) / \F_2}$ equals
\begin{align*}
\sum_{\vd \in \Div_{(5,5)}} \eta_{(5,5) / \F_2}(\vd) = 1+1+1+4 = 7,
\end{align*}
and the Wedderburn decomposition of $\CR_{(5,5) / \F_2}$ equals:
\begin{align*}
\CR_{(5,5) / \F_2} \cong \F_2 \oplus \F_{2^4} \oplus \F_{2^4} \oplus (\F_{2^4})^4 = \F_2 \oplus (\F_{2^4})^6.
\end{align*}

\section{Discussion}
We have shown the Wedderburn decomposition of semisimple commutative group algebras by investigating the decomposition of circulant rings, which serves as its geometric analogue.
A natural follow up research topic is to investigate the structure of commutative group algebras which are not semisimple, and whether one can obtain insights in their algebraic structure similar to their semisimple counterpart.

\section*{Acknowledgements}
I would like to thank Jan Schoone for providing this manuscript with useful feedback.

I would like to thank my PhD-supervisor prof. dr. Joan Daemen for providing me with research topics leading to this paper.

This work was supported by the European Research Council under the ERC advanced grant agreement under grant ERC-2017-ADG Nr.\ 788980 ESCADA.

\section*{Disclosure statement}
The author reports there are no competing interests to declare.



\bibliographystyle{elsarticle-num} 
\bibliography{reference}


%
%
%
\end{document}